\documentclass[11pt,a4paper]{amsart}
\usepackage{amssymb,amsmath,epsfig,graphics,mathrsfs,enumerate,verbatim}
\usepackage[pagebackref,colorlinks=true,linkcolor=blue,citecolor=blue]{hyperref}
\usepackage{fancyhdr}
\usepackage{hyperref}
\hypersetup{
 colorlinks   = true,
 urlcolor     = blue,
 linkcolor    = blue,
 citecolor   = red ,
 bookmarksopen=true
}

\usepackage{amsmath}
\usepackage{amsfonts}
\usepackage{amssymb}
\usepackage{amsthm}
\usepackage{epsfig,graphics,mathrsfs}
\usepackage{graphicx}
\usepackage{dsfont}

\usepackage[usenames, dvipsnames]{color}

\usepackage{hyperref}

\def \phi {\varphi}

\def \RN {\mathbb{R}^N}
\def \R {\mathbb{R}}

\def \G{\Gamma}
\newcommand{\rr}{{\mathbb{R}}}
\newcommand{\al}{\alpha}
\newcommand{\escpr}[1]{\langle#1\rangle}
\newcommand{\Ba}{\mathscr B_\alpha}

\def \vf{\varphi}

\newcommand{\Rn}{\mathbb R^n}
\newcommand{\Rm}{\mathbb R^m}

\newcommand{\Hn}{\mathbb H^n}

\newcommand{\p}{\partial}

\newcommand{\bG}{\mathbb {G}}

\newcommand{\la}{\lambda}

\numberwithin{equation}{section}

\newcommand{\beq}{\begin{equation}}
\newcommand{\bea}[1]{\begin{array}{#1} }
\newcommand{\eeq}{ \end{equation}}
\newcommand{\ea}{ \end{array}}

\newcommand{\ve}{\varepsilon}

\newcommand{\nh}{\nabla_H}
\newcommand{\sul}{\Delta_H}

\newcommand{\sa}{\langle}
\newcommand{\da}{\rangle}

\newcommand{\C}{\mathbb{C}}

\newtheorem{theorem}{Theorem}[section]
\newtheorem{lemma}[theorem]{Lemma}
\newtheorem{proposition}[theorem]{Proposition}
\newtheorem{corollary}[theorem]{Corollary}
\newtheorem{remark}[theorem]{Remark}
\newtheorem{definition}[theorem]{Definition}

\numberwithin{equation}{section}

\begin{document}

\title[A fundamental solution,  etc.]{A fundamental solution for a subelliptic operator in Finsler geometry}



\date{}

\begin{abstract}
We introduce a class of nonlinear partial differential equations in a product space which are at the interface of Finsler and sub-Riemannian geometry. To such equations we associate a non-isotropic Minkowski gauge $\Theta$ for which we introduce a suitable notion of Legendre transform $\Theta^0$. We compute the action of the relevant nonlinear PDEs on ``radial" functions, i.e., functions of $\Theta^0$, and by exploiting it we are able to compute explicit fundamental solutions of such PDEs. 
\end{abstract}

\author{Federica Dragoni}
\address{School of Mathematics\\
Cardiff University\\
Cardiff CF2 4AG
WALES} \email[Federica Dragoni]{DragoniF@cardiff.ac.uk}

\author{Nicola Garofalo}

\address{Dipartimento d'Ingegneria Civile e Ambientale (DICEA)\\ Universit\`a di Padova\\ Via Marzolo, 9 - 35131 Padova,  Italy}
\vskip 0.2in
\email{nicola.garofalo@unipd.it}

\author{Gianmarco Giovannardi}
\address{Department of Mathematics and Computer Science\\
Universit\`a di Firenze\\
50134 Florence,
ITALY}\email[Gianmarco Giovannardi]{gianmarco.giovannardi@unifi.it}

\author{Paolo Salani}
\address{Department of Mathematics and Computer Science\\
Universit\`a di Firenze\\
50134 Florence,
ITALY}\email[Paolo Salani]{paolo.salani@unifi.it}

\thanks{N. Garofalo is supported in part by a Progetto SID (Investimento Strategico di Dipartimento): ``Aspects of nonlocal operators via fine properties of heat kernels", University of Padova (2022); and by a PRIN (Progetto di Ricerca di Rilevante Interesse Nazionale) (2022): ``Variational and analytical aspects of geometric PDEs". He has also been partially supported by a Visiting Professorship at the Arizona State University. G. Giovannardi is supported in part by INdAM–GNAMPA 2023 Project \emph{Variational and non-variational problems with lack of compactness}. G. Giovannardi and P. Salani are supported in part by a PRIN (2022): "Geometric-Analytic Methods for PDEs and Applications (GAMPA)".}

\maketitle

\tableofcontents

\section{Introduction}\label{S:intro}

In this paper we compute explicit fundamental solutions for a new class of nonlinear partial differential operators which arise at the interface of Finsler and sub-Riemannian geometry. These two different geometries have developed independently of one another, but problems from the applied sciences (e.g., quasi-crystalline structures in gravitational physics, see \cite{V1}, \cite{V2}, \cite{BD} and the references therein) suggest that it is of interest merging them into a larger unifying body.  

Our starting point is the following prototypical model of a linear subelliptic partial differential operator of order two
in $\RN = \Rm\times \R^k$, with $z\in \Rm$, $\sigma\in \R^k$,   
\begin{equation}\label{bg}
\mathscr B_\alpha = \Delta_z + \frac{|z|^{2\alpha}}4 \Delta_\sigma,\ \ \ \ \ \ \ \ \ \ \ \ \ \ \ \ \alpha>0.
\end{equation}
Remarkably, $\Ba$ presents itself in several different areas, including free boundaries problems, analysis and geometry of CR manifolds, quasiconformal mappings, or the strong rigidity of locally symmetric spaces. For some of these aspects, the reader should see Section \ref{S:motback} below. The operator \eqref{bg} was first introduced by S. Baouendi in his doctoral dissertation \cite{B}. Subsequently, Grushin and Vishik  studied hypoellipticity questions in \cite{Gr1}, \cite{Gr2}, \cite{GV1}, \cite{GV2}. In \cite{Je} D. Jerison studied the solvability of the Dirichlet problem at characteristic points for the model case $m=k=\alpha= 1$ in \eqref{bg}. 
A seminal study of the local properties of weak solutions of equations modelled on \eqref{bg} was conducted in the mid 80's by Franchi and Lanconelli, see \cite{FLto}, \cite{FL}, \cite{FL84}, \cite{FL85}. There has been since a large body of works devoted to the many challenging aspects of the operator \eqref{bg}, and it would be impossible to provide a complete list here.  We confine ourselves to cite the papers \cite{G}, \cite{GS}, \cite{MM}, \cite{GRo} and the recent preprint \cite{BG}, as they directly exploit the function \eqref{Ga} below, and are in one way or another closely connected to the present one. In particular, in \cite{G} the second named author first discovered  the following explicit fundamental solution for $-\mathscr B_\alpha$ with singularity at $(0,0)$, and used such function to establish a basic monotonicity formula of Almgren type, see \cite[Theorem 4.2]{G}. Consider the anisotropic gauge defined by
\begin{align}\label{ra}
R_\alpha(z,\sigma) & = \left(|z|^{2(\alpha+1)}+4(\alpha+1)^2|\sigma|^2\right)^{\frac1{2(\alpha+1)}}.
\end{align}
Then, with a $C_\alpha>0$ explicitly given, 
 the function 
\begin{equation}\label{Ga}
 \Gamma_\alpha(z,\sigma)=\frac{C_{\alpha}}{R_\alpha(z,\sigma+\sigma')^{m + (\alpha+1) k-2}}
\end{equation}
is a fundamental solution for $-\mathscr B_\alpha$ with singularity at $(0,\sigma')\in \RN$.
Fundamental solutions and/or heat kernels with pole at arbitrary points, and for special values of $\alpha$, were constructed in \cite{BGG2}, \cite{BFIh}, \cite{BFI} and \cite{GTpotan}, but same are presently unknown for arbitrary $\alpha$ and singularities.

To introduce the questions of interest in the present work, consider the degenerate energy 
\begin{equation}\label{ebg}
E_{\alpha,p}(u) = \frac 1p \int_{\RN} \left(|\nabla_z u|^2 + \frac{|z|^{2\alpha}}4 |\nabla_\sigma u|^2\right)^{\frac p2} dzd\sigma,\ \ \ \ \ \ \ \ \ \ \ 1<p<\infty.
\end{equation}
We note that, when $p=2$, the Euler-Lagrange equation of \eqref{ebg} is $\mathscr B_\alpha u = 0$. When $p\not= 2$, degenerate energies such as \eqref{ebg} arise, e.g., in the foundational work of Kor\'anyi and H. M. Reimann \cite{KoR}, \cite{KoRaim}, and of Mostow \cite{Mo} and  Margulis and Mostow \cite{MaM}, see also \cite{CDGcpde}, \cite{D}, \cite{CDG} and \cite{HH}.

Suppose now that, more in general, on each of the two layers of $\RN= \Rm\times \R^k$ we assign strictly-convex  Minkowski norms 
$\Phi\in C^2(\Rm\setminus\{0\})$ and $\Psi\in C^2(\R^k\setminus\{0\})$, and respectively denote by $\Phi^0$ and $\Psi^0$ their Legendre transforms 
\begin{equation}\label{dual0}
\Phi^0(z) = \underset{\Phi(\zeta) = 1}{\sup}\ \sa z,\zeta\da, \ \ \ \ \ \ \ \ \Psi^0(\sigma) = \underset{\Phi(\eta) = 1}{\sup}\ \sa \sigma,\eta\da.
\end{equation}
The functions $\Phi^0$ and $\Psi^0$ are themselves norms, and their levels sets are often referred to as \emph{Wulff shapes}. 
We are interested in understanding critical points of the following generalisation of the energy \eqref{ebg}, 
\begin{equation}\label{Fen}
\mathscr E_{\alpha,p}(u) = \frac 1p \int_{\RN} \left(\Phi(\nabla_z u)^2 + \frac{\Phi^0(z)^{2\alpha}}4 \Psi(\nabla_\sigma u)^2\right)^{\frac p2} dzd\sigma,\ \ \ \ \ \ \ \ \ \ \ 1<p<\infty.
\end{equation}
It is clear that the special choice $\Phi(z) = |z|$, $\Psi(\sigma) = |\sigma|$, gives back \eqref{ebg}.
The Euler-Lagrange equation of \eqref{Fen} is the quasilinear PDE
\begin{equation}\label{Fel}
\mathscr L_{\alpha,p}(u) = \operatorname{div}_{(z,\sigma)} \left(\left(\Phi(\nabla_z u)^2 + \frac{\Phi^0(z)^{2\alpha}}4 \Psi(\nabla_\sigma u)^2\right)^{\frac{p-2}2} \mathscr A(\nabla_{(z,\sigma)} u)\right)\ =\ 0,
\end{equation}
where we have denoted
\begin{equation}\label{A}
\mathscr A(\nabla_{(z,\sigma)} u) = \begin{pmatrix} \Phi(\nabla_z u) \nabla \Phi(\nabla_z u)
\\
\frac{\Phi^0(z)^{2\alpha}}4 \Psi(\nabla_\sigma u) \nabla \Psi(\nabla_\sigma u)
\end{pmatrix}.
\end{equation}
In the special case in which $p=2$, the operator \eqref{Fel} becomes
\begin{equation}\label{Fel2}
\mathscr L_{\alpha,2} u = \Delta_{\Phi}(u) + \frac{\Phi^0(z)^{2\alpha}}4  \Delta_\Psi(u).
\end{equation}
In \eqref{Fel2} we have respectively indicated with $\Delta_{\Phi}$ and $\Delta_{\Psi}$ the Finsler Laplacians in $\Rm$ and $\R^k$ with respect to the norms $\Phi$ and $\Psi$, i.e., the operators 
\begin{equation}\label{Flap}
\Delta_\Phi(u) = \operatorname{div}_z(\Phi(\nabla_z u)\nabla \Phi(\nabla_z u)),\ \ \ \ \ \ \ \ \ \Delta_\Psi(u) = \operatorname{div}_\sigma(\Psi(\nabla_\sigma u)\nabla \Phi(\nabla_\sigma u)),
\end{equation}
see the appendix in Section \ref{S:app}.
 We emphasise that since each of these operators is nonlinear,  the same is true for the operator \eqref{Fel2}. The only choice that makes it linear, is $\Phi(z) = |z|$, $\Psi(\sigma) = |\sigma|$, in which case $\mathscr L_{\alpha,2}$ becomes the Baouendi-Grushin operator \eqref{bg}. 

To proceed in our discussion, we need to take a small detour. In sub-Riemannian geometry the tangent space is not Euclidean, but it is formed by a stratification of Euclidean spaces, each endowed with dilations weighted according to the relative position in the stratification. This anisotropy is typical of physical systems with non-holonomic constraints, in which motion is only allowed in certain directions prescribed by the physical problem at hand. The appropriate geometric setup is that of stratified nilpotent Lie groups, aka Carnot groups, or more in general, the Lie groups of homogeneous type in \cite{FS}. In this framework, Euclidean norms need to be replaced by an appropriately chosen \emph{anisotropic gauge} that weighs different directions accordingly. For instance, the Euclidean model \eqref{ra} above is one-homogeneous with respect to the following family of anisotropic dilations
\begin{equation}\label{dil}
\delta_t(z,\sigma) = (t z, t^{\alpha+1} \sigma),\ \ \ \ \ \ \ t>0.
\end{equation}
This suggests that with the energy \eqref{Fen} we should associate the following  anisotropic \emph{Minkowski gauge} in $\RN$ 
\begin{equation}\label{theta}
\Theta(z,\sigma) = \left(\Phi(z)^{2(\alpha+1)} + 4(\alpha+1)^2 \Psi(\sigma)^2\right)^{\frac{1}{2(\alpha+1)}},
\end{equation}
for which we clearly have $\Theta(\delta_t(z,\sigma)) = t \Theta(z,\sigma)$.
Now, in Finsler geometry the dual Minkowski norm (see \eqref{dual} in Section \ref{S:app}) plays a central role. If one wants to introduce such geometry in a sub-Riemannian setting, one is immediately confronted with the fact that the classical Legendre transformation does not work for a non-isotropic gauge such as \eqref{theta}. To overcome such hindrance, we introduce a new Legendre transformation $\Theta^0$, 
namely
\begin{equation}\label{theta0}
\Theta^0(z,\sigma)^{\alpha+1}= \underset{\Theta(\xi,\tau) = 1}{\sup}\ \bigg(|\sa z,\xi\da|^{\alpha+1}+ 4(\alpha+1)^2 \sa\sigma, \tau\da\bigg).
\end{equation}
The remarkable feature of \eqref{theta0} is underscored by Proposition \ref{P:newL} below, in which we solve the relevant Lagrange multiplier problem, and prove the notable property that
\begin{equation}\label{theta00}
\Theta^0(z,\sigma)= \left(\Phi^0(z)^{2(\al +1)}+  4(\alpha+1)^2 \Psi^0(\sigma)^2  \right)^{\frac{1}{2(\al+1)}},
\end{equation}
where $\Phi^0$ and $\Psi^0$ are the classical Legendre transforms in \eqref{dual0}.
The reader should note the striking symmetry between \eqref{theta} and \eqref{theta00}. From this perspective, one should think of \eqref{ra} above as a special instance of \eqref{theta00}, except that this important distinction disappears when $\Phi(z) = |z|$ and $\Psi(\sigma) = |\sigma|$, since in such case $\Phi^0(z) = \Phi(z)$ and $\Psi(\sigma) = \Psi^0(\sigma)$.

\medskip
 
With \eqref{theta00} in hands, we can now state the first main result in this paper. Henceforth, we denote
\begin{equation}\label{Q}
Q = Q_\alpha = m + (\alpha+1) k.
\end{equation}
Observe that since Lebesgue measure in $\RN$ scales according to the formula 
\[
d(\delta_t(z,\sigma)) = t^Q dzd\sigma,
\]
such number plays the role of a dimension. Also, for brevity we will from now on in this paper denote $\rho(z,\sigma) = \Theta^0(z,\sigma)$ the dual anisotropic Minkowski gauge \eqref{theta00}.

\begin{theorem}\label{T:halleyop}
Let $\alpha>0$ and $p>1$. Given $F\in C^2(\R^+)$, we have in $\RN\setminus\{0\}$
\begin{align}\label{lollypop}
\mathscr L_{\alpha,p}(F\circ \rho) & = (p-1) |F'(\rho)|^{p-2} \left\{F''(\rho) + \frac{Q-1}{p-1}\frac{F'(\rho)}{\rho}\right\}\left(\frac{\Phi^0(z)}{\rho}\right)^{\alpha p}.\end{align}
In particular, the function
\begin{equation}\label{cand}
u = \begin{cases}
\rho^{-\frac{Q-p}{p-1}},\ \ \ \ \ \ \ \ \ \ \ p\not= Q,
\\
\log \rho,\ \ \ \ \ \ \ \ \ \ \ \ \ p = Q,
\end{cases}
\end{equation}
is a solution of $\mathscr L_{\alpha,p}(u) = 0$ in $\RN\setminus\{0\}$. 
\end{theorem}

In the remark that follows, we explicitly underline an interesting and perhaps unexpected consequence of Theorem \ref{T:halleyop}. 

\begin{remark}\label{R:p2}
When $p = 2$, formula \eqref{lollypop} reads
\begin{align}\label{lollypoppino}
\mathscr L_{\alpha,2}(F\circ \rho) & = \left\{F''(\rho) + \frac{Q-1}{\rho} F'(\rho)\right\}\left(\frac{\Phi^0(z)}{\rho}\right)^{2\alpha }.
\end{align}
Since the right-hand side of \eqref{lollypoppino} is linear on functions $F\circ \rho$, it follows that, despite its above noted strong nonlinearity, the operator $\mathscr L_{\alpha,2}$ acts linearly on functions of the anisotropic dual gauge \eqref{theta00}.
\end{remark}

\begin{remark}\label{R:fed}
We emphasise that the distortion factors $\left(\frac{\Phi^0(z)}{\rho}\right)^{\alpha p}$ in \eqref{lollypop}, or  $\left(\frac{\Phi^0(z)}{\rho}\right)^{2\alpha }$ in \eqref{lollypoppino}, are typically sub-Riemannian phenomena, see Lemmas \ref{L:en} and \ref{L:L2rho} below. In Euclidean Finsler geometry such factors are not present, see Corollary \ref{C:supernice} in Section \ref{S:app} below. Without such factors one would for instance infer that $\mathscr L_{\alpha,2}$ of a function of $\rho = \Theta^0$ is a function of $\Theta^0$, a symmetry property which is totally false in sub-Riemannian geometry, even in the case $\Phi(z) = |z|$, $\Psi(\sigma) = |\sigma|$!
\end{remark}

The second main result of the present paper is the following.

\begin{theorem}\label{T:wow}
Given $\alpha>0$ and $1<p<\infty$, let 
\begin{equation}\label{Cip}
\omega_{\alpha,p} = \int_{\{\rho<1\}} \left(\frac{\Phi^0(z)}{\rho}\right)^{\alpha p} dz d\sigma,\ \ \ \ \ \ \sigma_{\alpha,p} = Q \omega_{\alpha,p}.
\end{equation}
Define $C_{\alpha,p}>0$ as follows
\begin{equation}\label{Cippy}
C_{\alpha,p} = \begin{cases} 
\frac{p-1}{Q-p} \left(\sigma_{\alpha,p}\right)^{-1/(p-1)},\ \ \ \ \ p\not= Q,
\\
\\
\sigma_{\alpha,Q}^{-1/(Q-1)},\ \ \ \ \ \ \ \ \ \ p = Q.
\end{cases}
\end{equation}
Then the function
\begin{equation}\label{wow}
\mathscr G_{\alpha,p}(z,\sigma) = \begin{cases}
C_{\alpha,p}\ \rho(z,\sigma)^{-\frac{Q-p}{p-1}},\ \ \ \ \ \ \ \ p\not= Q,
\\
\\
C_\alpha\ \log \rho(z,\sigma),\ \ \ \ \ \ \ \ \ \ \ \ p = Q,
\end{cases}
\end{equation}
is a fundamental solution  of $-\mathscr L_{\alpha,p}$ with pole in $(0,0)$.
\end{theorem}
 
\begin{remark}\label{R:pug}
Since the operator $\mathscr L_{\alpha,p}$ is invariant under Euclidean translations along the manifold $M = \{0\}_z\times \R^k$, by taking $(z,\sigma)\to \mathscr G_{\alpha,p}(z,\sigma+\sigma_0)$ we obtain from Theorem \ref{T:wow} a fundamental solution with pole at any point of $(0,\sigma_0)\in M$.
\end{remark}

\begin{remark}\label{R:gia}
When $\Phi(z) = |z|$ and $\Psi(\sigma) = |\sigma|$, Theorem \ref{T:wow} produces the following fundamental solution
\begin{equation}\label{Gap}
 \Gamma_{\alpha,p}(z,\sigma)=\frac{C_{\alpha,p}}{R_\alpha(z,\sigma)^{\frac{Q-p}{p-1}}}
\end{equation}
 of the nonlinear operator
\[
\operatorname{div}_{(z,\sigma)}\left(\left(|\nabla_z u|^2 + \frac{|z|^{2\alpha}}4 |\nabla_\sigma u|^2\right)^{\frac{p-2}2} \begin{pmatrix} \nabla_z u \\ |z|^{2\alpha} \nabla_\sigma u\end{pmatrix}\right).
\] 
In \eqref{Gap} we have denoted by $R_\alpha(z,\sigma)$ the gauge in \eqref{ra} above. 
\end{remark}
It is worth noting that, even in this specialised context, this result is new (the case $p=2$ was found in \cite{G}, as previously mentioned), and that \eqref{Gap} displays an interesting stability with respect to the parameter of subellipticity $\alpha>0$. By this we mean that if we denote $x=(z,\sigma)\in \RN$, and  let $\alpha\to 0^+$, then in view of \eqref{Cip}, \eqref{Cippy} the function in \eqref{Gap} precisely converges to the well-known
\[
\Gamma_{p}(x)=\frac{C_{N,p}}{|x|^{\frac{N-p}{p-1}}},
\]
fundamental solution of the classical \emph{$p$-Laplacian} $\Delta_p(u) = \operatorname{div}(|\nabla u|^{p-2} \nabla u)$ in $\RN$.  On the other hand, when $\alpha\to 1$, we have $Q\to m+2k$ in \eqref{Q}, and from \eqref{ra} we obtain
\[
R_\alpha(z,\sigma)\ \longrightarrow\ N(z,\sigma) = (|z|^4 + 16 |\sigma|^2)^{1/4},
\]
the Kor\'anyi-Kaplan gauge in $\bG$. We thus recover from \eqref{Gap} the fundamental solution of the horizontal $p$-Laplacian $\Delta_{H,p}(u) = \operatorname{div}_H(|\nh u|^{p-2}\nh u)$ in a group of Heisenberg type $\bG$ discovered in \cite[Theorem 2.1]{CDG}
\begin{equation*}
\Gamma_p(x) = \begin{cases}
\frac{C_p}{N(x)^{\frac{Q-p}{p-1}}},\ \ \ \ \text{when}\ p\not=Q,
\\
\\
C_p \log N(x),\ \ \ \ \ \text{when}\ p=Q,
\end{cases}
\end{equation*}
As it is well-known, the case $p=2$ of this result was first discovered by Folland in \cite{Fo} (when $\bG = \Hn$, the Heisenberg group), and generalised by Kaplan to any group of Heisenberg type in \cite[Theorem 2]{Ka}. For such groups, the case $p = Q$ was also independently found in \cite{HH}.

Some final comments are in order. S. S. Chern has been the great architect of Finsler geometry, which he dubbed as ``Riemannian geometry without the quadratic restriction", see \cite{Chernnot}. This means that the relevant manifold is endowed with a smoothly varying family of Minkowski norms in the tangent space which replace the Euclidean ones. It is from this enriching perspective that Finsler geometry can be viewed as a generalisation of the Riemannian one, see \cite{Chernfinsler} and \cite{Chern}. Its mathematical appeal, and its interest in the applied sciences, see e.g. \cite{Tay}, have caused the subject to undergo a profound development during the past three decades. For an introduction, the reader is referred to the seminal studies \cite{BP}, \cite{BCS}, \cite{BR}, \cite{CeSh}, \cite{Sh1} and \cite{Sh2}. For some beautiful recent developments we refer the reader to \cite{OScpam}, \cite{OSaim}, \cite{Moo}, \cite{MY}, and the references therein. In the framework of PDEs in Finsler spaces, there is a large, growing literature, and it would be impossible to list it here. We confine ourselves to mention the work \cite{CS} by Cianchi and the third named author which treats an overdetermined Serrin type problem regarding the Finsler Laplacian. A generalisation of their results to Finsler Monge-Amp\`ere equations is contained in the recent paper \cite{CSjfa} (see also the bibliography of these works). 

Sub-Riemannian geometry is also an extension of Riemannian geometry. One wants to model phenomena with non-holonomic constraints, when motion at any point is only allowed along a limited set of directions which are prescribed by the physical problem at hand. When the set of directions coincides with the whole tangent space, one obtains Riemannian geometry. The foundations of this geometry were laid in E. Cartan's address \cite{Ca} at the 1928 ICM in Bologna. Nowadays, sub-Riemannian geometry and the closely connected theory of subelliptic partial differential equations have grown into a full fledged area of investigation. 

Interestingly, in 1934 E. Cartan also introduced in Finsler geometry the connection that bears his name, see \cite{Ca34} (Cartan's connection is metric compatible, but it is not Levi-Civita since it is not torsion-free. A different, torsion-free connection was discovered in 1948 by S. S. Chern, although his connection is not metric compatible. For these aspects the interested reader should consult \cite{BCS}). 

As far as we are aware of, there has been no attempt so far in merging these two fascinating areas of geometry into a unified body. In this paper our starting point is the energy \eqref{Fen}, a prototypical model in which different norms are assigned on different layers of the stratification. A seemingly different situation was introduced in the Ph.D. thesis \cite{Sa}, and further developed in the subsequent papers \cite{PR}, \cite{FMRS}, \cite{GR}, \cite{GR2}, \cite{GPR} and \cite{GPPV}. These works consider an interesting generalisation of the notion of horizontal perimeter first set forth in \cite{cag}, but unlike the geometric case $p=1$ in \eqref{Fen} above, in the relevant ``energy" horizontal and vertical variables are inextricably mixed.


\section{Motivational background}\label{S:motback}

The purpose of this section is twofold: (i) to provide the reader with some perspective on the many roles of the Baouendi-Grushin operator \eqref{bg} in analysis and geometry; (ii) to stimulate further understanding of its Finsler generalisation \eqref{Fel}.
    
\subsection{The lifting theorem} We begin by recalling that in  the opening of their celebrated \emph{lifting} paper (see \cite[Example (b) on p.149]{RS}), Rothschild and Stein discuss the model differential operator in the plane with coordinates $(z,\sigma)$,
\begin{equation}\label{1}
\mathscr B u = \p_{zz} u + z^2 \p_{\sigma\sigma} u.
\end{equation}
It is clear that, letting $X_1 = \p_z, X_2 = z \p_\sigma$, one has $\mathscr B =  X_1^2  + X_2^2$, and that $X_1$ and $[X_1,X_2] = \p_\sigma$ span the tangent space. Although there exist no non-commutative nilpotent Lie groups of dimension two, one can remedy this by adding one extra variable $x$. In this way the vector fields $X_1$ and $X_2$ can be lifted to the three-dimensional Heisenberg group $\mathbb H^1$, with generators of the Lie algebra $\tilde X_1 = \p_z$ and $\tilde X_2 = \p_x + z\p_\sigma$. Since $[\tilde X_1,\tilde X_2] = \p_\sigma$, one immediately obtains the hypoellipticity of $\mathscr B$ in $\R^2$, from that of the horizontal Laplacian $\Delta_H = \tilde X_1^2 + \tilde X_2^2$ in $\R^3\cong \mathbb H^1$. The latter in fact follows from H\"ormander's famous theorem in \cite{Ho}, or more directly from Folland's explicit fundamental solution in \cite{Fo}.   
This provides a first glimpse of the close ties between the model operator $\mathscr B$ and the horizontal Laplacian $\Delta_H$ on the Heisenberg group. 

\subsection{Lie groups of Heisenberg type} In a different direction, consider a Lie group of Heisenberg type $\bG$. In the logarithmic coordinates $(z,\sigma)$ with respect to a fixed orthonormal basis of the Lie algebra, the horizontal Laplacian is given by
\begin{equation}\label{slH}
\Delta_H = \Delta_z + \frac{|z|^2}{4} \Delta_\sigma  + \sum_{\ell = 1}^k \p_{\sigma_\ell} \sum_{i<j} b^\ell_{ij} (z_i \p_{z_j} -  z_j \p_{z_i}),
\end{equation}
where $b^\ell_{ij}$ indicate the group constants given by $[e_i,e_j] = \sum_{\ell=1}^k b^\ell_{ij} \ve_\ell$, see e.g. \cite[Section 2.5]{Gems}. Here, $z = \sum_{i=1}^m z_i e_i \in \Rm$ represents the variable point in the horizontal layer of the Lie algebra, whereas $\sigma =\sum_{\ell = 1}^k \sigma_\ell \ve_\ell\in \R^k$ indicates the variable point in the center of $\bG$. 
 When $u(z,\sigma) = f(|z|,\sigma)$ is a function with cylindrical symmetry, then  because of the skew-symmetry of $b^\ell_{ij}$, we have $\sum_{i<j} b^\ell_{ij} (z_i \p_{z_j} -  z_j \p_{z_i}) u = 0$, and therefore the action of $\sul$ on such $u$ is given by  
\begin{equation}\label{2}
\sul u = \Delta_z u + \frac{|z|^2}4 \Delta_\sigma u.
\end{equation}
This provides another instance of the multi-faced link between horizontal Laplacians on nilpotent Lie groups of Heisenberg type and degenerate elliptic operators such as \eqref{1}, or its generalisation \eqref{2}. 

\subsection{Weakly pseudo-convex domains} A third connection comes from CR geometry. Operators such as \eqref{2} arise in the study of the boundary Cauchy-Riemann complex $\Box_b$. As it is well-known,  the Heisenberg group $\Hn$ can be naturally identified with the boundary of the Siegel upper half-space in $\mathbb C^{n+1}$, i.e., the domain
\[
\mathscr D^{n+1}_+ = \{z = (z_1,...,z_{n+1})\in \mathbb C^{n+1}\mid \Im(z_{1})>\sum_{j=2}^{n+1} |z_j|^2\},
\]
see \cite{FS74}. The domain $\mathscr D^{n+1}_+$ is strictly pseudo-convex, i.e. the Levi form on its boundary is  strictly positive definite. If for $p\in \mathbb N$, $p>1$, we consider instead the weakly pseudo-convex domain 
\[
\mathscr D_{p,+} = \{z = (z_1,z_{2})\in \mathbb C^{2}\mid \Im(z_{1})> |z_2|^{2p}\}
\]
(see p. 132 in \cite{DT} and forward for a computation of the Levi form),
 then $\p \mathscr D_{p,+}$ can no longer be endowed with a group structure. Up to a renormalisation factor, on functions which are ``radial" in the variable $z_2$, the sub-Laplacian on the boundary of   $\mathscr D_{p,+}$ is a special case of \eqref{bg} above.

 For instance, when $p=2$, then a fundamental solution of the sub-Laplacian on the boundary of $\mathscr D_{2,+}$ was found by Greiner in \cite[Theorem 7]{Gre}. When the pole is at a point $(0,\sigma')$, then such fundamental solution is given by the formula
\[
\G_2((z,\sigma),(0,\sigma')) = \frac{4}{\pi} (|z|^8 + (\sigma+\sigma')^2)^{-1/2}.
\]
In such case $m=2$, $k=1$, $\alpha = 3$, and the number defined by \eqref{Q} above is $Q  = 6$. The reader should note that, up to a rescaling factor, the function $\G_2$ coincides with the one given by the general formula  \eqref{Ga}. The geometric situation in \cite{Gre} was subsequently generalised by Beals, Gaveau and Greiner, who considered real hypersurfaces in $\mathbb C^2$ which include the boundaries of the pseudo-Siegel domains $\mathscr D_{p,+}$, with $p\in \mathbb N$, see \cite[Theorem 1.1]{BGG1}. In such general framework, one has $\alpha = 2p-1$, and therefore \eqref{Q}  presently gives $Q = 2p+2$. At a point $(0,\sigma')$, formula (1.19) of \cite{BGG1} gives the following fundamental solution 
\[
\G_p(z,\sigma) =  \frac{\pi^2}{4p} (|z|^{4p} + (\sigma+\sigma')^2)^{-1/2},
\]
which is again (up to a rescaling factor) a special case of \eqref{Ga}. 

\subsection{Free boundary problems} In this subsection we discuss the connection of \eqref{bg} with  free boundary problems for the fractional Laplacian. Consider the pseudodifferential operator in $\Rn$ defined on the Fourier transform side by the formula $\widehat{(-\Delta)^s u}(\xi) = (2\pi|\xi|)^{2s} \hat u(\xi)$. In their celebrated paper \cite{CaS} Caffarelli and Silvestre have shown that, when $0<s<1$, this nonlocal operator can be recovered as a weighted  Dirichlet-to-Neumann map of the following \emph{extension problem} in $\R^{n+1}_+ = \Rn_\sigma\times \R^+_y$
\begin{equation}\label{cs}
L_a \tilde U = \operatorname{div}_{(\sigma,y)}(y^a \nabla_{(\sigma,y)} \tilde U) = y^a(\Delta_\sigma \tilde U + \tilde U_{yy} + \frac ay \tilde U_y) = 0,\ \ \ \ \ \tilde U(\sigma,0) = u(\sigma),\ \ \ \ \ \sigma\in \Rn,
\end{equation}
with $a = 1-2s\in (-1,1)$. But, in fact, $(-\Delta)^s$ is also in a natural way the (unweighted) Dirichlet-to-Neumann map of an operator such as \eqref{bg} above. To see this, for a given function $\tilde U(\sigma,y)$, consider the change of variable $ U(\sigma,z) = \tilde U(\sigma,y)$, where $z>0$ and $y>0$ are given by the relation
\[
z = \left(\frac{y}{1-a}\right)^{1-a},\ \ \ \text{or equivalently},\ \ \ \ y = (1-a) z^{\frac{1}{1-a}}.
\]
Then the Caffarelli-Silvestre extension operator $L_a$ is connected to the Baouendi-Grushin operator \eqref{bg} by the following formula
\begin{equation}\label{labb2}
y^a L_a  \tilde U(\sigma,y) = D_{zz}  U + z^{2\alpha} \Delta_\sigma  U,
\end{equation}
where $\alpha = \frac{a}{1-a}$. Using the equation \eqref{labb2}, one can express the extension theorem in \cite{CaS} in the following equivalent fashion, see \cite[Proposition 11.2]{Gft}. Given $s\in (0,1)$, let $\alpha = \frac{1}{2s} -1\in (-\frac 12,\infty)$. If for $u\in \mathscr S(\Rn)$ one considers the solution $U(x,z)$ to the Dirichlet problem
\begin{equation}\label{bg00}
\begin{cases}
D_{zz} U(x,z) + z^{2\alpha}  \Delta_x U(x,z)  = 0,\ \ \ \ (x,z)\in \R^{n+1}_+,
\\
U(x,0) = u(x),
\end{cases}
\end{equation}
then one has
\begin{equation}\label{dtnbg}
(-\Delta)^s u(x) = - \frac{\G(1+s)}{\G(1-s)}\ \underset{z\to 0^+}{\lim} \frac{\p U}{\p z}(x,z).
\end{equation}
Using \eqref{labb2} and \eqref{bg00}, one can directly extract from the above cited monotonicity formula in \cite[Theorem 4.2]{G} the one for the operator $L_a$ in \cite[Theorem 3.1 (see Remark 3.2)]{CSS}.  

\subsection{Conformal geometry} 
A final motivating example for \eqref{bg} originates from conformal CR geometry. In the paper \cite{BFM} Branson, Fontana and Morpurgo introduced a  pseudodifferential operator $\mathscr L_{s}$ in the Heisenberg group $\Hn$ which is the correct geometric counterpart of the fractional powers $(-\Delta)^s$. We emphasise in this respect that the fractional powers $(-\Delta_H)^s$ of the Kohn-Spencer horizontal Laplacian in $\Hn$ are not conformal, and they have no geometric significance (for the definition of $\Delta_H$ in any group of Heisenberg type see \eqref{slH} above). In their work \cite{FGMT} Frank, del Mar Gonz\'alez,  Monticelli and Tan introduced and solved a CR counterpart for the operator $\mathscr L_{s}$ of the Caffarelli-Silvestre extension problem \eqref{cs} (but very different from it!), see also the works of Roncal and Thangavelu \cite{RT}, \cite{RT2} for a parabolic version of it. In the previously cited paper \cite{GTpotan}, the authors constructed the following explicit heat kernel for the Frank, del Mar Gonz\'alez, Monticelli and Tan extension problem in any group of Heisenberg type
\begin{align}\label{hk}
q_{(s)}((z,\sigma),t,y) & =  \frac{2^k}{(4\pi t)^{\frac{m}2 +k +1-s}} \int_{\R^k} e^{- \frac it \langle \sigma,\la\rangle}   \left(\frac{|\la|}{\sinh |\la|}\right)^{\frac m2+1-s}
\\
\notag
& \times e^{-\frac{|z|^2 +y^2}{4t}\frac{|\la|}{\tanh |\la|}} d\la.
\notag
\end{align}
The construction of \eqref{hk} ultimately hinged on the heat kernel (in a space with fractal dimension) for the case $\alpha = 1$ of the Baouendi-Grushin operator \eqref{bg}. In this respect,  
the reader should compare the ``dimension" $m+2k + 2(1-s)$ in the factor $t^{\frac{m}2 +k +1-s}$ in \eqref{hk}  with the case $\alpha =1$ of \eqref{Q} above. 



\section{The Anisotropic Legendre transformation and its dual}\label{S:aniBG}

As it is well-known, in Finsler geometry the classical Legendre transform of a Minkowski norm plays a central role, see \eqref{dual} below. Unfortunately, because of the different scalings in different directions of the tangent space, such transformation is not well-suited for the geometric framework of this paper.
As we have mentioned in Section \ref{S:intro}, when attempting to combine Finsler with sub-Riemannian geometry, the first question that arises is what is an appropriate notion of Legendre transformation for a non-isotropic Minkowski gauge on product spaces. 
In this section we introduce a new notion of Legendre transformation which is tailor-made to deal with this problem. We focus our attention on the model situation of a Euclidean space $\RN= \rr^m \times \rr^k$ with $m,k\ge1$, and $N = m+k$, but we emphasise that our results hold unchanged for product spaces with an arbitrary number of factors. Let $x=(z,\sigma) \in \RN$ where $z \in \rr^m$ and $\sigma \in \rr^k$. Given $\alpha >0$ we consider a family of anisotropic dilations $\{\delta_t\}_{t \in \rr^+}$ given by \eqref{dil} above.
Let $\Phi: \rr^m \to [0,+\infty)$ and $\Psi: \rr^k \to [0,\infty)$ be two Minkowski norms on  $\rr^m$ and $\rr^k$ respectively, see Section \ref{S:app} for the precise definition.

\begin{definition}\label{D:animink}
We define an \emph{anisotropic Minkowski gauge} on $\RN$ by the equation  \eqref{theta} above.
\end{definition}
It is worth observing right-away that $\Theta$ is positively one-homogeneous with respect to the anisotropic dilations \eqref{dil}, i.e., 
\[
\Theta(\delta_t(z,\sigma)) = t\ \Theta(z,\sigma).
\]
  In the next definition we introduce a new anisotropic Legendre transformation which, as Theorem \ref{T:halleyop}, Remark \ref{R:p2} and Theorem \ref{T:wow} show, displays some remarkable properties. Its first basic feature is given in Proposition \ref{P:newL} below.

\begin{definition}\label{D:newL} 
For a fixed point $(z,\sigma)\in \RN$, we define the \emph{anisotropic Legendre transformation} $\Theta^0$ of the gauge $\Theta$ by the equation \eqref{theta0}.
\end{definition}

We have the following basic result.

\begin{proposition}\label{P:newL}
The solution $\Theta^0$ of the constrained extremum problem \eqref{theta0} is identified by the equation 
\[
\Theta^0(z,\sigma)= \left(\Phi^0(z)^{2(\al +1)}+  4(\alpha+1)^2 \Psi^0(\sigma)^2  \right)^{\frac{1}{2(\al+1)}},
\]
where $\Phi^0$ and $\Psi^0$ respectively denote the standard Legendre transforms \eqref{dual0} of the Minkowski norms $\Phi$ and $\Psi$.
\end{proposition}

\begin{proof}
Let $(\xi, \tau)$ be a constrained extremum point for the problem \eqref{theta0} under the constraint
\[
\Theta(z,\sigma)^{2(\alpha+1)} = \Phi(\xi)^{2(\alpha+1)}+ 4(\alpha+1)^2 \Psi(\tau)^2 = 1.
\]
By the method of Lagrange multipliers we must have at such point 
\begin{equation}
\label{eq:LMsystem}
\begin{cases}
|\escpr{z,\xi}|^{\al} \text{sign} \escpr{z,\xi}\ z= 2\ \la\ \Phi(\xi)^{2\al+1} \nabla_{\xi} \Phi(\xi),
\\
\sigma= 2\ \lambda\   \Psi(\tau) \nabla_{\tau} \Psi(\tau),
\\
\Phi(\xi)^{2(\alpha+1)}+ 4(\alpha+1)^2 \Psi(\tau)^2=1.
\end{cases}
\end{equation}
To understand \eqref{eq:LMsystem}, we first take the inner product with $\xi$ in the former equation, and with $\tau$ in the second one, and then we apply  \eqref{hom} to both $\Phi$ and $\Psi$. We obtain  
\begin{equation}\label{eq:LMSy2}
\begin{cases}
|\escpr{z,\xi}|^{\al+1} = 2 \la \Phi(\xi)^{2(\al+1)},
\\
\escpr{\sigma,\tau}= 2 \la  \Psi(\tau)^2,
 \\
\Phi(\xi)^{2(\alpha+1)}+ 4(\alpha+1)^2 \Psi(\tau)^2=1.
\end{cases}
\end{equation}
We now add the first line in the previous system to the second one multiplied by  $4(\alpha+1)^2$. Using the third line (the constraint), and keeping \eqref{theta0} in mind, we find 
\begin{equation}
\label{eq:la}
\Theta^0(z,\sigma)^{\alpha+1}=2 \la.
\end{equation}
Substitution of \eqref{eq:la} in the system \eqref{eq:LMsystem} yields
\begin{equation}
\label{eq:LMsystem2}
\begin{cases}
|\escpr{z,\xi}|^{\al} \text{sign} \escpr{z,\xi}\ z= \Theta^0(z,\sigma)^{\alpha+1}\ \Phi(\xi)^{2\al+1} \nabla_{\xi} \Phi(\xi),
\\
\sigma= \Theta^0(z,\sigma)^{\alpha+1}\   \Psi(\tau) \nabla_{\tau} \Psi(\tau),
\\
\Phi(\xi)^{2(\alpha+1)}+ 4(\alpha+1)^2 \Psi(\tau)^2=1.
\end{cases}
\end{equation}
We now apply $\Phi^0 \nabla \Phi^0$ to both sides of the first equation,  and $\Psi^0 \nabla \Psi^0$ to both sides of the second equation, obtaining
\begin{equation*}
\begin{cases}
|\escpr{z,\xi}|^{\al} \text{sign} \escpr{z,\xi}\  \Phi^0(z) \nabla \Phi^0(z)= \Theta^0(z,\sigma)^{\alpha+1} \Phi(\xi)^{2\al+1}\Phi^0(\nabla_\xi \Phi(\xi)) \nabla \Phi^0(\nabla \Phi(\xi)),
\\
 \Psi^0(\sigma) \nabla \Psi^0(\sigma)= \Theta^0(z,\sigma)^{\alpha+1} \Psi(\tau) \Psi^0(\nabla_\tau \Psi(\tau)) \nabla \Psi^0(\nabla \Psi(\tau)),
 \\
 \Phi(\xi)^{2(\alpha+1)}+ 4(\alpha+1)^2 \Psi(\tau)^2=1.
\end{cases}
\end{equation*}
Using \eqref{Finabla} and Lemma \ref{L:BP} for both $\Phi$ and $\Psi$, we thus find
\begin{equation*}
\begin{cases}
|\escpr{z,\xi}|^{\al} \text{sign} \escpr{z,\xi}\  \Phi^0(z) \nabla \Phi^0(z)= \Theta^0(z,\sigma)^{\alpha+1} \Phi(\xi)^{2\al}\ \xi,
\\
 \Psi^0(\sigma) \nabla \Psi^0(\sigma)= \Theta^0(z,\sigma)^{\alpha+1} \ \tau,
 \\
 \Phi(\xi)^{2(\alpha+1)}+ 4(\alpha+1)^2 \Psi(\tau)^2=1.
\end{cases}
\end{equation*}
In both sides of the first equation we now take the inner product with $z$, and we do the same with $\sigma$ in the second equation. By the homogeneity of $\Phi^0$, $\Psi^0$, and by Euler equation, we find 
\begin{equation*}
\begin{cases}
 \Phi^0(z)^2= \Theta^0(z,\sigma)^{\alpha+1} \Phi(\xi)^{2\al} |\escpr{z,\xi}|^{1-\alpha},
 \\
 \Psi^0(\sigma)^2= \Theta^0(z,\sigma)^{\alpha+1}  \escpr{\tau,\sigma},
 \\
 \Phi(\xi)^{2(\alpha+1)}+ 4(\alpha+1)^2 \Psi(\tau)^2=1.
\end{cases}
\end{equation*}
From equations \eqref{eq:LMSy2} and \eqref{eq:la}, we have
\begin{equation*}
\begin{cases}
 \Phi^0(z)^2= \Theta^0(z,\sigma)^2  \Phi(\xi)^2,
 \\
 \Psi^0(\sigma)^2= \Theta^0(z,\sigma)^{2(\alpha+1)}  \Psi(\tau)^2,
 \\
 \Phi(\xi)^{2(\alpha+1)}+ 4(\alpha+1)^2 \Psi(\tau)^2=1.
\end{cases}
\end{equation*}
Raising to the power $(\alpha+1)$ the first equation, adding the second equation multiplied by $4(\alpha+1)^2$, and using the constraint, we finally obtain
\[
\Theta^0(z,\sigma)= \left(\Phi^0(z)^{2(\al +1)}+  4(\alpha+1)^2 \Psi^0(\sigma)^2  \right)^{\frac{1}{2(\al+1)}},
\]
which gives the desired conclusion.

\end{proof}


\section{Proof of Theorems \ref{T:halleyop} and \ref{T:wow}}\label{S:proof}

In this section we prove Theorems \ref{T:halleyop} and \ref{T:wow}. As said right before the statement of Theorem \ref{T:halleyop}, we indicate with $\rho(z,\sigma) = \Theta^0(z,\sigma)$ the dual anisotropic Minkowski gauge \eqref{theta00}.
We begin with a key lemma which provides a form of \emph{eikonal} equation.

\begin{lemma}\label{L:en}
We have in $\RN\setminus\{0\}$
\[
\Phi(\nabla_z \rho)^2 + \frac{\Phi^0(z)^{2\alpha}}4 \Psi(\nabla_\sigma \rho)^2 = \left(\frac{\Phi^0(z)}{\rho}\right)^{2\alpha}.
\]
\end{lemma}

\begin{proof}
For $\sigma\in \R^k$ fixed, we consider the function
\begin{equation}\label{f}
f(s) = \left(s^{2(\alpha+1)} + 4(\alpha+1)^2 \Psi^0(\sigma)^2\right)^{\frac{1}{2(\alpha+1)}}.
\end{equation}
An elementary computation shows that
\begin{equation}\label{fder}
\begin{cases}
f'(s) = s^{2\alpha+1} f(s)^{-2\alpha-1},
\\
f''(s) = 4(2\alpha+1)(\alpha+1)^2 \Psi^0(\sigma)^2 s^{2\alpha} f(s)^{-4\alpha-3}.
\end{cases}
\end{equation}
Similarly, for $z\in \Rm$ fixed we consider the function
\begin{equation}\label{g}
g(t) = \left(\Phi^0(z)^{2(\alpha+1)} + 4(\alpha+1)^2\ t^2\right)^{\frac{1}{2(\alpha+1)}},
\end{equation}
and obtain
\begin{equation}\label{gder}
\begin{cases}
g'(t) = 4(\alpha+1) t\ g(t)^{-2\alpha-1},
\\
g''(t) = 4(\alpha+1)  \left(\Phi^0(z)^{2(\alpha+1)} - 4\alpha(\alpha+1) t^2\right) g(t)^{-4\alpha-3}.
\end{cases}
\end{equation}
Since
\[
\rho(z,\sigma) = f(\Phi^0(z)) = g(\Psi^0(\sigma)),
\]
the chain rule and \eqref{fder},\eqref{gder} thus give
\begin{equation}\label{cr1}
\nabla_z \rho = \Phi^0(z)^{2\alpha+1} \rho^{-2\alpha-1}\nabla_z \Phi^0(z),\ \ \ \ \ \ \nabla_\sigma \rho = 4(\alpha+1)\Psi^0(\sigma)\rho^{-2\alpha-1} \nabla_\sigma \Psi^0(\sigma).
\end{equation}
Applying, respectively, $\Phi$ and $\Psi$ to the latter equations, and using \eqref{Finabla}, we find 
\begin{equation}\label{cr2}
\Phi(\nabla_z \rho) = \Phi^0(z)^{2\alpha+1} \rho^{-2\alpha-1},\ \ \ \ \Psi(\nabla_\sigma \rho) = 4(\alpha+1)\Psi^0(\sigma)\rho^{-2\alpha-1}.
\end{equation}
This gives
\begin{align*}
& \Phi(\nabla_z \rho)^2 + \frac{\Phi^0(z)^{2\alpha}}4 \Psi(\nabla_\sigma \rho)^2 = \rho^{-4\alpha-2}\left\{\Phi^0(z)^{4\alpha+2}  + \frac{\Phi^0(z)^{2\alpha}}4 (4(\alpha+1))^2\Psi^0(\sigma)^2\right\}
\\
& = \Phi^0(z)^{2\alpha}\rho^{-4\alpha-2}\left\{\Phi^0(z)^{2(\alpha+1)} + 4(\alpha+1)^2 \Psi^0(\sigma)^2\right\} 
\\
& =  \Phi^0(z)^{2\alpha}\rho^{-4\alpha-2} \rho^{2(\alpha+1)} = \left(\frac{\Phi^0(z)}{\rho}\right)^{2\alpha},
\end{align*}
which completes the proof.

\end{proof}

\begin{remark}\label{R:hommanipadmeum}
We explicitly note that the factor $\left(\frac{\Phi^0(z)}{\rho}\right)^{2\alpha}$ in the right-hand side of the equation in Lemma \ref{L:en} has homogeneity zero with respect to the anisotropic dilations \eqref{dil}. 
\end{remark}

We will also need the following notable fact. To motivate it, we recall that if $r(x) = |x|$, with $x\in \Rn$, then it is well known that
\begin{equation}\label{curvatures}
\Delta r = \frac{n-1}r,
\end{equation}
and this equation plays a pervasive role both in pde's and geometry. Now, the operator $\mathscr L_{\alpha,2}$ defined by \eqref{Fel2} is strongly nonlinear. It is quite remarkable that, despite this aspect, its action on the anisotropic dual Finsler norm $\rho = \Theta^0$ is expressed by the following intrinsic replacement of \eqref{curvatures}, see also Remark \ref{R:hommanipadmeum}. 

\begin{lemma}\label{L:L2rho}
For the operator in \eqref{Fel2} we have in $\RN\setminus\{0\}$
\[
\mathscr L_{\alpha,2}(\rho) = \frac{Q-1}{\rho} \left(\frac{\Phi^0(z)}{\rho}\right)^{2\alpha},
\]
where $Q$ is as in \eqref{Q}.
\end{lemma}

\begin{proof}
According to \eqref{Fel2}, we have in $\RN\setminus\{0\}$
\[
\mathscr L_{\alpha,2}(\rho) = \Delta_{\Phi}(\rho) + \frac{\Phi^0(z)^{2\alpha}}4  \Delta_\Psi(\rho).
\]
We now consider the functions $f$ and $g$ introduced in \eqref{f} and \eqref{g}. From Corollary \eqref{C:supernice} we obtain
\[
\Delta_\Phi(\rho) = f''(\Phi^0) + \frac{m-1}{\Phi^0} f'(\Phi^0),\ \ \ \Delta_\Psi(\rho) = g''(\Psi^0) + \frac{k-1}{\Psi^0} g'(\Psi^0).
\]
Using \eqref{fder}, \eqref{gder} in the latter two equations, we find
\begin{equation}\label{Delphirho}
\Delta_\Phi(\rho) = (m+2\alpha)(\Phi^0)^{2\alpha} \rho^{-2\alpha-1}-(2\alpha+1)(\Phi^0)^{4\alpha+2}\rho^{-4\alpha-3},
\end{equation}
and
\begin{equation}\label{Delpsirho}
\Delta_\Psi(\rho) = 4(\alpha+1)\ k\ \rho^{-2\alpha-1}-(2\alpha+1)(4(\alpha+1))^2 (\Psi^0)^{2}\rho^{-4\alpha-3}.
\end{equation}
Combining \eqref{Delphirho} with \eqref{Delpsirho}, after some elementary computations we obtain
\begin{align*}
\mathscr L_{\alpha,2}(\rho) & = \frac{(\Phi^0)^{2\alpha}}{\rho^{2\alpha+1}} (m+2\alpha + (\alpha+1)k) 
 - (2\alpha+1) \frac{(\Phi^0)^{2\alpha}}{\rho^{4\alpha+3}} \left[(\Phi^0)^{2(\alpha +1)} + 4(\alpha+1)^2 (\Psi^0)^2\right].
\end{align*}
Since by \eqref{Q} we have $m+2\alpha + (\alpha+1)k = Q+2\alpha$, and since Proposition \ref{P:newL} gives $(\Phi^0)^{2(\alpha +1)} + 4(\alpha+1)^2 (\Psi^0)^2 = \rho^{2(\alpha+1)}$, we easily obtain the desired conclusion from the latter equation.

\end{proof}

We are ready to present the
 
\begin{proof}[Proof of Theorem \ref{T:halleyop}]

Given a function $F\in C^2(\R^+)$, we intend to compute the action of the operator in \eqref{Fel} on a function $u = F(\rho)$. Since
\[
\nabla_z u = F'(\rho) \nabla_z \rho,\ \ \ \ \ \nabla_\sigma u = F'(\rho) \nabla_\sigma \rho,
\]
we have from Lemma \ref{L:en} in $\RN\setminus\{0\}$
\begin{equation}\label{encomp}
\Phi(\nabla_z u)^2 + \frac{\Phi^0(z)^{2\alpha}}4 \Psi(\nabla_\sigma u)^2 = F'(\rho)^2 \left\{\Phi(\nabla_z \rho)^2 + \frac{\Phi^0(z)^{2\alpha}}4 \Psi(\nabla_\sigma \rho)^2\right\} = F'(\rho)^2 \left(\frac{\Phi^0(z)}{\rho}\right)^{2\alpha}.
\end{equation}
On the other hand, \eqref{A} gives
\begin{equation}\label{Arho}
\mathscr A(\nabla_{(z,\sigma)} u) = F'(\rho) \begin{pmatrix} \Phi(\nabla_z \rho) \nabla \Phi(\nabla_z \rho)
\\
\frac{\Phi^0(z)^{2\alpha}}4 \Psi(\nabla_\sigma \rho) \nabla \Psi(\nabla_\sigma \rho)
\end{pmatrix}.
\end{equation}
Using \eqref{encomp} and \eqref{Arho} in \eqref{Fel}, we obtain
{\allowdisplaybreaks
\begin{align*}
\mathscr L_{\alpha,p}(u) & = \operatorname{div}_{(z,\sigma)} \left(F'(\rho) |F'(\rho)|^{p-2} \left(\frac{\Phi^0(z)}{\rho_\alpha}\right)^{\alpha(p-2)} \mathscr A(\nabla_{(z,\sigma)} \rho)\right)
\\
& = \operatorname{div}_{z} \left(F'(\rho) |F'(\rho)|^{p-2} \left(\frac{\Phi^0(z)}{\rho_\alpha}\right)^{\alpha(p-2)} \Phi(\nabla_z \rho) \nabla \Phi(\nabla_z \rho)\right)
\\
& + \operatorname{div}_{\sigma} \left(F'(\rho) |F'(\rho)|^{p-2} \left(\frac{\Phi^0(z)}{\rho_\alpha}\right)^{\alpha (p-2)} \frac{\Phi^0(z)^{2\alpha}}4 \Psi(\nabla_\sigma \rho) \nabla \Psi(\nabla_\sigma \rho)\right)
\\
& = F'(\rho) |F'(\rho)|^{p-2} \left(\frac{\Phi^0(z)}{\rho}\right)^{\alpha(p-2)} \left\{ \Delta_{\Phi}(\rho) + \frac{\Phi^0(z)^{2\alpha}}4 \Delta_\Psi(\rho)\right\}
\\
& + \Phi(\nabla_z \rho) \left\sa\nabla_z \Phi(\nabla_z \rho),\nabla_z \left(F'(\rho) |F'(\rho)|^{p-2} \left(\frac{\Phi^0(z)}{\rho}\right)^{\alpha(p-2)}\right)\right\da
\\
& + \frac{\Phi^0(z)^{2\alpha}}4 \Psi(\nabla_\sigma \rho) \left\sa\nabla \Psi(\nabla_\sigma \rho), \nabla_\sigma \left(F'(\rho) |F'(\rho)|^{p-2} \left(\frac{\Phi^0(z)}{\rho}\right)^{\alpha (p-2)}\right) \right\da.
\end{align*}}
Using Lemma \ref{L:L2rho}, we have in $\RN\setminus\{0\}$
\begin{align}\label{un}
& F'(\rho) |F'(\rho)|^{p-2} \left(\frac{\Phi^0(z)}{\rho}\right)^{\alpha(p-2)} \left\{ \Delta_{\Phi}(\rho) + \frac{\Phi^0(z)^{2\alpha}}4 \Delta_\Psi(\rho)\right\}
\\
& = \frac{Q-1}{\rho}  F'(\rho) |F'(\rho)|^{p-2} \left(\frac{\Phi^0(z)}{\rho}\right)^{\alpha p}.
\notag
\end{align}
Next, in the above expression of $\mathscr L_{\alpha,p}(u)$ we evaluate the term 
\begin{align*}
& \left(\frac{\Phi^0(z)}{\rho}\right)^{\alpha (p-2)}\bigg\{\Phi(\nabla_z \rho) \sa\nabla \Phi(\nabla_z \rho),\nabla_z \left(F'(\rho) |F'(\rho)|^{p-2}\right)\da
\\
& + \frac{\Phi^0(z)^{2\alpha}}4 \Psi(\nabla_\sigma \rho) \sa\nabla \Psi(\nabla_\sigma \rho), \nabla_\sigma \left(F'(\rho) |F'(\rho)|^{p-2} \right) \da\bigg\}
\\
& = (p-1) |F'(\rho)|^{p-2} F''(\rho) \left(\frac{\Phi^0(z)}{\rho}\right)^{\alpha (p-2)}\bigg\{\Phi(\nabla_z \rho) \sa\nabla \Phi(\nabla_z \rho),\nabla_z \rho\da
\\
& +  \frac{\Phi^0(z)^{2\alpha}}4 \Psi(\nabla_\sigma \rho) \sa\nabla \Psi(\nabla_\sigma \rho), \nabla_\sigma \rho \da\bigg\}
\\
& = (p-1) |F'(\rho)|^{p-2} F''(\rho) \left(\frac{\Phi^0(z)}{\rho}\right)^{\alpha (p-2)}\bigg\{\Phi(\nabla_z \rho)^2 +   \frac{\Phi^0(z)^{2\alpha}}4 \Psi(\nabla_\sigma \rho)^2\bigg\}.
\end{align*}
Applying Lemma \ref{L:en}, we conclude that 
\begin{align}\label{due}
& \left(\frac{\Phi^0(z)}{\rho}\right)^{\alpha (p-2)}\bigg\{\Phi(\nabla_z \rho) \sa\nabla \Phi(\nabla_z \rho),\nabla_z \left(F'(\rho) |F'(\rho)|^{p-2}\right)\da
\\
& + \frac{\Phi^0(z)^{2\alpha}}4 \Psi(\nabla_\sigma \rho) \sa\nabla \Psi(\nabla_\sigma \rho), \nabla_\sigma \left(F'(\rho) |F'(\rho)|^{p-2} \right) \da\bigg\}
\notag\\
& = (p-1) |F'(\rho)|^{p-2} F''(\rho) \left(\frac{\Phi^0(z)}{\rho}\right)^{\alpha p}.
\notag
\end{align}
Finally, we analyse the remaining term in the expression of $\mathscr L_{\alpha,p}(u)$. We make the following claim:
\begin{align}\label{tre}
& \Phi(\nabla_z \rho) \left\sa\nabla_z \Phi(\nabla_z \rho),\nabla_z \left(\frac{\Phi^0(z)}{\rho}\right)^{\alpha(p-2)}\right\da
\\
& + \frac{\Phi^0(z)^{2\alpha}}4 \Psi(\nabla_\sigma \rho) \left\sa\nabla \Psi(\nabla_\sigma \rho), \nabla_\sigma \left(\frac{\Phi^0(z)}{\rho}\right)^{\alpha (p-2)}\right\da\ = \ 0.
\notag
\end{align}
If for a moment we take the claim for granted, it should be clear to the reader that, by combining \eqref{un}, \eqref{due}, \eqref{tre}, we obtain
\begin{align*}
\mathscr L_{\alpha,p}(u) & = \frac{Q-1}{\rho}  F'(\rho) |F'(\rho)|^{p-2} \left(\frac{\Phi^0(z)}{\rho}\right)^{\alpha p} + (p-1) |F'(\rho)|^{p-2} F''(\rho) \left(\frac{\Phi^0(z)}{\rho}\right)^{\alpha p}
\\
& = (p-1) |F'(\rho)|^{p-2} \left\{F''(\rho) + \frac{Q-1}{p-1}\frac{F'(\rho)}{\rho}\right\}\left(\frac{\Phi^0(z)}{\rho}\right)^{\alpha p},
\end{align*}
which would  complete the proof of the theorem. We are thus left with proving \eqref{tre}. With this in mind, we have 
{\allowdisplaybreaks
\begin{align*}
& \Phi(\nabla_z \rho) \left\sa\nabla_z \Phi(\nabla_z \rho),\nabla_z \left(\frac{\Phi^0(z)}{\rho}\right)^{\alpha(p-2)}\right\da
\\
& + \frac{\Phi^0(z)^{2\alpha}}4 \Psi(\nabla_\sigma \rho) \left\sa\nabla \Psi(\nabla_\sigma \rho), \nabla_\sigma \left(\frac{\Phi^0(z)}{\rho}\right)^{\alpha (p-2)}\right\da
\\
& = \Phi^0(z)^{\alpha (p-2)}\bigg\{\Phi(\nabla_z \rho) \sa\nabla_z \Phi(\nabla_z \rho),\nabla_z \rho^{\alpha(2-p)}\da
\\
& +  \frac{\Phi^0(z)^{2\alpha}}4 \Psi(\nabla_\sigma \rho) \sa\nabla \Psi(\nabla_\sigma \rho), \nabla_\sigma \rho^{\alpha (2-p)}\da\bigg\}
\\
& + \alpha(p-2)\rho^{\alpha(2-p)} \Phi^0(z)^{\alpha(p-2)-1} \Phi(\nabla_z \rho) \sa\nabla_z \Phi(\nabla_z \rho),\nabla_z \Phi^0(z)\da
\\
& = - \alpha(p-2)\rho^{\alpha(2-p)-1} \Phi^0(z)^{\alpha (p-2)} \bigg\{\Phi(\nabla_z \rho)^2 +   \frac{\Phi^0(z)^{2\alpha}}4 \Psi(\nabla_\sigma \rho)^2\bigg\}
\\
& + \alpha(p-2)\rho^{\alpha(2-p)} \Phi^0(z)^{\alpha(p-2)-1} \rho^{-2\alpha-1} \Phi^0(z)^{2\alpha+1} \sa\nabla_z \Phi(\nabla_z \Phi^0(z)),\nabla_z \Phi^0(z)\da
\\
& = - \alpha(p-2)\rho^{\alpha(2-p)-1} \Phi^0(z)^{\alpha (p-2)} \left(\frac{\Phi^0(z)}{\rho}\right)^{2\alpha} 
+ \frac{\alpha(p-2)}{\rho}\left(\frac{\Phi^0(z)}{\rho}\right)^{\alpha p}
\\
& = - \frac{\alpha(p-2)}{\rho}\left(\frac{\Phi^0(z)}{\rho}\right)^{\alpha p} + \frac{\alpha(p-2)}{\rho}\left(\frac{\Phi^0(z)}{\rho}\right)^{\alpha p}\ =\ 0, 
\end{align*}}
where in the third to the last equality we have used \eqref{cr1}, \eqref{cr2}, and in the second to the last equality we have used Lemma \ref{L:en} and the fact that $\sa\nabla_z \Phi(\nabla_z \Phi^0(z)),\nabla_z \Phi^0(z)\da = \Phi(\nabla_z \Phi^0(z)) = 1$, by Euler formula and Lemma \ref{L:en} again. This proves \eqref{tre}, and we are done.

\end{proof}

We next give the 

\begin{proof}[Proof of Theorem \ref{T:wow}]

Keeping \eqref{Fel} in mind, we intend to show that, for any $\vf\in C^\infty_0(\RN)$, we have
\begin{equation}\label{zab}
\int_{\RN}  \left(\Phi(\nabla_z \mathscr G_{\alpha,p})^2 + \frac{\Phi^0(z)^{2\alpha}}4 \Psi(\nabla_\sigma \mathscr G_{\alpha,p})^2\right)^{\frac{p-2}2} \sa \mathscr A(\nabla_{(z,\sigma)} \mathscr G_{\alpha,p}),\nabla_{(z,\sigma)} \vf\da dzd\sigma\ =\ \vf(0),
\end{equation}
where in view of \eqref{A} we have
\[
\mathscr A(\nabla_{(z,\sigma)} \mathscr G_{\alpha,p}) = \begin{pmatrix} \Phi(\nabla_z \mathscr G_{\alpha,p}) \nabla \Phi(\nabla_z \mathscr G_{\alpha,p})
\\
\frac{\Phi^0(z)^{2\alpha}}4 \Psi(\nabla_\sigma \mathscr G_{\alpha,p}) \nabla \Psi(\nabla_\sigma \mathscr G_{\alpha,p})
\end{pmatrix}.
\]
We observe explicitly that the function $\mathscr G_{\alpha,p}\notin L^1_{loc}(\RN)$, unless $p>\frac{2Q}{Q+1}$. However, in order to make sense of \eqref{zab} we do not need this integrability, but rather that 
\[
\left(\Phi(\nabla_z \mathscr G_{\alpha,p})^2 + \frac{\Phi^0(z)^{2\alpha}}4 \Psi(\nabla_\sigma \mathscr G_{\alpha,p})^2\right)^{\frac{p-2}2} \left|\mathscr A(\nabla_{(z,\sigma)} \mathscr G_{\alpha,p})\right|\ \in\ L^1_{loc}(\RN).
\]
This is guaranteed by the fact that
\[
\left(\Phi(\nabla_z \mathscr G_{\alpha,p})^2 + \frac{\Phi^0(z)^{2\alpha}}4 \Psi(\nabla_\sigma \mathscr G_{\alpha,p})^2\right)^{\frac{p-2}2} \left|\mathscr A(\nabla_{(z,\sigma)} \mathscr G_{\alpha,p})\right|\ \cong\ \rho^{1-Q}.
\]
Going back to \eqref{zab}, we have 
\begin{align*}
& \int_{\RN}  \left(\Phi(\nabla_z \mathscr G_{\alpha,p})^2 + \frac{\Phi^0(z)^{2\alpha}}4 \Psi(\nabla_\sigma \mathscr G_{\alpha,p})^2\right)^{\frac{p-2}2} \sa \mathscr A(\nabla_{(z,\sigma)} \mathscr G_{\alpha,p}),\nabla_{(z,\sigma)} \vf\da dzd\sigma
\\
& =  \underset{\ve\to 0^+}{\lim}\ \int_{\{\rho>\ve\}}  \left(\Phi(\nabla_z \mathscr G_{\alpha,p})^2 + \frac{\Phi^0(z)^{2\alpha}}4 \Psi(\nabla_\sigma \mathscr G_{\alpha,p})^2\right)^{\frac{p-2}2} \sa \mathscr A(\nabla_{(z,\sigma)} \mathscr G_{\alpha,p}),\nabla_{(z,\sigma)} \vf\da dzd\sigma
\\
& = \underset{\ve\to 0^+}{\lim}\ \int_{\{\rho=\ve\}}  \left(\Phi(\nabla_z \mathscr G_{\alpha,p})^2 + \frac{\Phi^0(z)^{2\alpha}}4 \Psi(\nabla_\sigma \mathscr G_{\alpha,p})^2\right)^{\frac{p-2}2} \sa \mathscr A(\nabla_{(z,\sigma)} \mathscr G_{\alpha,p}),\nu\da\ \vf\  dH_{N-1},
\end{align*}
where in the last equality we have used the divergence theorem, Theorem \ref{T:halleyop}, and the fact that $\vf$ is compactly supported. In the last integral, $\nu$ indicates the outer unit normal to the anisotropic Wulff sphere $\{\rho=\ve\}$, and $dH_{N-1}$ the $(N-1)$-dimensional Hausdorff measure in $\RN$ restricted to such sphere.
To continue our discussion, we now assume that $1<p<\infty$, but $p\not= Q$. The case $p = Q$, is dealt with similarly, and we leave the relevant details to the interested reader. If we write
$\mathscr G_{\alpha,p} = C \rho^{-\frac{Q-p}{p-1}}$, with $C>0$ to be determined, by Lemma \ref{L:en} a computation gives 
\[
\left(\Phi(\nabla_z \mathscr G_{\alpha,p})^2 + \frac{\Phi^0(z)^{2\alpha}}4 \Psi(\nabla_\sigma \mathscr G_{\alpha,p})^2\right)^{\frac{p-2}2} = C^{p-2} \left(\frac{Q-p}{p-1}\right)^{p-2} \rho^{\frac{(1-Q)(p-2)}{p-1}} \left(\frac{\Phi^0}{\rho}\right)^{\alpha(p-2)}.
\]
Since on $\{\rho=\ve\}$ we have
\[
\nu = - \frac{\nabla_{(z,\sigma)} \rho}{|\nabla_{(z,\sigma)} \rho|},
\]
another computation gives
\[
\sa \mathscr A(\nabla_{(z,\sigma)} \mathscr G_{\alpha,p}),\nu\da = C \frac{Q-p}{p-1} \rho^{\frac{1-Q}{p-1}} \left(\frac{\Phi^0}{\rho}\right)^{2\alpha}\frac{1}{|\nabla_{(z,\sigma)} \rho|}.
\]
Combining quantities, we obtain
\begin{align*}
& \int_{\{\rho=\ve\}}  \left(\Phi(\nabla_z \mathscr G_{\alpha,p})^2 + \frac{\Phi^0(z)^{2\alpha}}4 \Psi(\nabla_\sigma \mathscr G_{\alpha,p})^2\right)^{\frac{p-2}2} \sa \mathscr A(\nabla_{(z,\sigma)} \mathscr G_{\alpha,p}),\nu\da\ \vf\  dH_{N-1} 
\\
& = C^{p-1} \left(\frac{Q-p}{p-1}\right)^{p-1} \ve^{1-Q} \int_{\{\rho=\ve\}}  \vf\  \left(\frac{\Phi^0}{\rho}\right)^{\alpha(p-2)} \frac{dH_{N-1}}{|\nabla_{(z,\sigma)} \rho|}\ \underset{\ve\to 0^+}{\longrightarrow}\ \vf(0),
\end{align*}
provided that
\begin{equation}\label{C}
C = \sigma_{\alpha,p}^{-\frac{1}{p-1}} \frac{p-1}{Q-p},
\end{equation}
where
\begin{equation}\label{sig}
\sigma_{\alpha,p} = \int_{\{\rho=1\}}\left(\frac{\Phi^0}{\rho}\right)^{\alpha p} \frac{dH_{N-1}}{|\nabla_{(z,\sigma)} \rho|}
\end{equation}

\end{proof}


\section{Appendix: known facts}\label{S:app}

Let $N:\Rn\to [0,\infty)$ be a Minkowski norm in $\Rn$. By this we mean that $N^2\in \C^{2}(\Rn\setminus\{0\})$ is a strictly convex function such that $N(\la x) = |\la| N(x)$ for every $x\in \Rn$ and $\la\in \R$. Since all norms in $\Rn$ are equivalent, there exist constants $\beta \ge \alpha>0$ such that
\begin{equation}\label{equinorm}
\alpha |\xi| \le N(\xi) \le \beta |\xi|.
\end{equation} 
We denote by  
\begin{equation}\label{dual}
N^0(x) = \underset{N(\xi)=1}{\sup}\ \sa x,\xi\da,
\end{equation}
its Legendre transform, also known as the dual norm of $N$. The Cauchy-Schwarz inequality holds 
\begin{equation}\label{CS}
|\sa x,y\da| \le N(x) N^0(y).
\end{equation}
A basic property of the norms $N$ and $N^0$ is the following, see \cite[Lemma 2.1]{BP}, and also (3.12) in \cite{CS},
\begin{equation}\label{Finabla}
N(\nabla N^0(x)) = N^0(\nabla N(x))  = 1,\ \ \ \ \ \ \ \ x\in \Rn\setminus\{0\}.
\end{equation}
This identity says  that for any $x\in \Rn\setminus\{0\}$ the vector $\nabla \Phi^0(x)\in \p K = S_1$, and that equivalently $\nabla \Phi(x)\in \p K^0 = S^0_1$. Given a function $u\in C^1(\Rn)$, an elementary, yet useful,  consequence of the homogeneity of $\Phi$ is
\begin{equation}\label{hom}
\sa \nabla N(\nabla u(x)),\nabla u(x)\da = N(\nabla u(x)).
\end{equation}
A less obvious but basic fact is the following formula, which can be found in \cite[Lemma 2.2]{BP}.

\begin{lemma}\label{L:BP}
For every $x\in \Rn\setminus\{0\}$, one has
\[
N^0(x) \nabla N(\nabla N^0(x)) = x,\ \ \ \ \ \ N(x) \nabla N^0(\nabla N(x)) = x. 
\]
\end{lemma}

Consider the energy
\begin{equation}\label{finen}
\mathscr E_{N}(u) = \frac 12 \int N(\nabla u)^2 dx.
\end{equation}
The Euler-Lagrange equation of \eqref{finen} is the so-called Finsler Laplacian
\begin{equation}\label{EL}
\Delta_N(u) = \operatorname{div}(N(\nabla u)\nabla N(\nabla u)) = 0.
\end{equation}
It is worth emphasising here that the operator in \eqref{EL} is quasilinear, but not linear, unless of course $N(x) = |x|$. However, the operator $\Delta_N$ is elliptic. In fact, since $N^2$ is homogeneous of degree $2$, Euler formula gives
\[
\sa\nabla(N^2)(\xi),\xi\da = N(\xi)^2,
\]
and from \eqref{equinorm} we thus have for every $\xi\in \Rn$
\[
\alpha^2 |\xi|^2 \le \sa\nabla(N^2)(\xi),\xi \da \le \beta^2 |\xi|^2.
\]
Recall that if a function $\Psi:\Rn\setminus\{0\}\to \R$ satisfies for some $\kappa\in \R$ the homogeneity condition
\[
\Psi(\la x) = |\la|^\kappa \Psi(x),\ \ \ \ \ \ \la\not= 0,
\]
then for its gradient one has
\[
\nabla \Psi(\la x) = \frac{|\la|}{\la} |\la|^{\kappa-1} \nabla \Psi(x).
\]  
With this observation in mind, we have the following useful chain rule.

\begin{lemma}\label{L:chain}
Let $u\in C^2(\Rn)$, $h\in C^2(\R)$. Then
\[
\Delta_N(h\circ u) = h'(u) \Delta_N(u) + h''(u) N(\nabla u)^2.
\]
\end{lemma}

A basic fact is the following result which can be found in \cite[Theorem 2.1]{FK}. 
\begin{theorem}\label{T:nice}
Let $f(x) = \frac{N^0(x)^2}2$.
Then $\Delta_N f(x) = n$ for every $x\in \Rn$.
\end{theorem}

A notable consequence of Theorem \ref{T:nice} is the following result that shows the remarkable fact that, on functions of the dual norm $N^0$, the nonlinear operator $\Delta_N$ acts linearly. In this respect, one should see the interesting paper \cite{AIS}.

\begin{corollary}\label{C:supernice}
If $k\in C^2(\R)$, one has in $\Rn\setminus\{0\}$
\[
\Delta_N(k\circ N^0) = k''(N^0) + \frac{n-1}{N^0} k'(N^0).
\]
\end{corollary}

\bibliographystyle{amsplain}

\end{document}